\documentclass[10pt]{amsart}

\usepackage{amsfonts,amssymb,amscd,amstext}
\usepackage[latin1]{inputenc}

\usepackage[paper size={210mm,297mm},left=34mm,right=34mm,top=43.5mm,bottom=43.5mm]{geometry}
\usepackage[charter]{mathdesign}
\usepackage[colorlinks=true,linkcolor=blue,citecolor=red]{hyperref}
\usepackage{graphicx}

\usepackage{esint}
\usepackage{enumerate}

\renewcommand{\le}{\leqslant}
\renewcommand{\ge}{\geqslant}
\newcommand{\ptl}{\partial}

\newcommand{\rr}{{\mathbb{R}}}

\newcommand{\G}{{\mathbb{G}}}

\newcommand{\hh}{{\mathbb{H}}}
\newcommand{\nn}{{\mathbb{N}}}

\newcommand{\hhh}{\mathcal{H}}

\newcommand{\escpr}[1]{g(#1)}
\newcommand{\Sg}{\Sigma}

\newcommand{\Om}{\Omega}
\newcommand{\eps}{\varepsilon}

\newcommand{\ga}{\gamma}
\newcommand{\Ga}{\Gamma}

\newcommand{\mnh}{|N_{H}|}

\newcommand{\n}{\nabla}

\DeclareMathOperator{\divv}{div}
\DeclareMathOperator{\tor}{Tor}

\newtheorem{theorem}{Theorem}[section]
\newtheorem{proposition}[theorem]{Proposition}

\newtheorem{corollary}[theorem]{Corollary}

\theoremstyle{definition}

\newtheorem{definition}{Definition} 

\theoremstyle{remark}

\newtheorem{remark}[theorem]{Remark}

\numberwithin{equation}{section}

\begin{document}

\title[The regularity of Euclidean Lipschitz boundaries with prescribed mean curvature]{The regularity of  Euclidean Lipschitz boundaries with prescribed \\mean curvature in three-dimensional contact\\ sub-Riemannian manifolds}

\author[M.~Galli]{Matteo Galli} \address{Dipartimento di Matematica \\
Universit\`a di Bologna \\ Piazza di Porta S. Donato 5 \\ 40126 Bologna, Italy}
\email{matteo.galli8@unibo.it}

\date{\today}

\thanks{M. Galli has been supported by the People Programme (Marie Curie Actions) of the European Union's Seventh
Framework Programme FP7/2007-2013/ under REA grant agreement n. 607643.}

\begin{abstract}
In this paper we consider a set $E\subset\Om$ with prescribed mean curvature $f\in C(\Om)$ and Euclidean Lipschitz boundary $\partial E=\Sg$ inside a three-dimensional contact sub-Riemannian manifold $M$. We prove that if $\Sg$ is locally a regular intrinsic graph, the characteristic curves are of class $C^2$.  The result is shape and improves the ones contained in \cite{MR2583494} and \cite{GalRit15}. 
\end{abstract}

\maketitle

\bibliographystyle{alpha}

\thispagestyle{empty}

\section{Introduction}

The study of the sub-Riemannian area minimizers start with the seminal paper \cite{MR1404326}, where general properties and existence of sets with minimum perimeter are proved in Carnot groups. In the last two decades many effort are been devoted to the develop of an exhaustive theory for minimal surfaces in the sub-Riemannian setting. One of the main difficulties related with this propose  is to provide a regularity theory for critical points of the sub-Riemannian area functional.  

Nowadays in sub-Riemannian geometry the optimal regularity of critical points of the area is not clear. More precisely, there are only partial results  under some assumptions on the regularity of the area-stationary surfaces themselves. Many author focused their attention on the study of minimal $C^2$ surfaces or isoperimetric sets with $C^2$ boundary, see \cite{MR2354992}, \cite{MR2648078}, \cite{MR2472175}, \cite{MR2333095}, \cite{MR2455341},  \cite{MR2435652},  \cite{MR2609016},  \cite{MR2165405} and \cite{MR2043961}, among others. A very clear introduction on this first results is presented in the monograph \cite{MR2312336}. 

The first examples of non-smooth solutions of the Plateau problem with smooth boundary datum in a regular domain are constructed in \cite{MR2225631}. Later on, in \cite{MR2262784} and \cite{MR2448649}, the authors find examples of entire graphs over the $xy$-plane, the so called $t$-graphs, that are area-minimizers and Euclidean Lipschitz. On the other hand, we stress that these examples are not $\hh$-regular graphs in the sense of  \cite{MR1871966} and are smooth in their regular part. Another remarkable result is the existence of a non-entire Euclidean Lipschitz $\hh$-regular graph that is a global minimizers on a strip, \cite[Example~2.8]{MR2333095}.

In \cite{MR2262784} and \cite{MR2481053} it is studied  the existence and the regularity of $C^1$ $t$-graphs in $\hh^n$ that are critical points of the functional 
\begin{equation}
\label{eq:F}
\mathcal{F}(u)=\int_\Om |\nabla u+\vec{F}|+\int_\Om fu,
\end{equation}
on a domain $\Om\subset\rr^{2n}$, where $\vec{F}$ is a vector field and $f\in L^\infty(\Om)$. When $\vec{F}(x,y)=(-y,x)$, the integral $\int_\Om |\nabla u+\vec{F}|$ is the sub-Riemannian area of the graph over the domain $\Om$ in the $xy$-plane. Among several interesting results, they prove in \cite[Thm.~A]{MR2481053} that, in case $n=1$, $u\in C^1(\Om)$ is an stationary point of $\mathcal{F}$ and $f\in C^0(\Om)$, the characteristic curves are of class $C^2$.

Theorem~A in \cite{MR2481053} is well-known for $C^2$ minimizers and generalizes a previous result by Pauls \cite[Lemma~3.3]{MR2225631} for $H$-minimal surfaces with components of the horizontal Gauss map in the class $W^{1,1}$. For an explication of what horizontal means we refer to Definition~\ref{def:horizontal}. 

In \cite{MR2583494} the authors consider a Euclidean Lipschitz minimal intrinsic graph of a function $u$ over a domain $\Om$ inside a vertical plane. 
Under the assumption that $u$ is a vanishing viscosity minimal graphs (it means that there exists a family $\{u_j\}_{j\in \nn}$ of smooth functions $u_j:\Om\rightarrow \rr$ converging uniformly to $u$ on compact subsets of $\Om$ and such that the graph of $u_j$ is a Riemannian minimal surface),  they prove an intrinsic $C_\hh^{1,\alpha}$ regularity for the graph of $u$. As a consequence they obtain that the graph of $u$ is foliated by horizontal smooth parabolas. 

Very recently, in \cite{GalRit15} the authors consider the sub-Riemannian prescribed mean curvature functional
\begin{equation}
\label{eq:JOm}
\mathcal{J}(E,B)=P(E,B)+\int_{E\cap B} f,
\end{equation}
where $E$ is a set of locally finite sub-Riemannian perimeter in $\Om$, $P(E,B)$ is the relative sub-Riemannian perimeter of $E$ in a bounded open set $B\subset\Om$, and $f\in L^\infty(\Om)$. If $E\subset\hh^n$ is the subgraph of a function $t=u(x,y)$ in the Heisenberg group $\hh^n$, then $\mathcal{J}(E)$  coincides with \eqref{eq:F} taking $\vec{F}(x,y)=(-y,x)$.
In \cite{GalRit15}, see also \cite{gr-aim}, they prove that an intrinsic graph of a $C^1$ function that is a critical point of the prescribed mean curvature functional  is foliated by smooth curves in a 3-dimensional contact manifold. Such curves, usually called characteristic curves, are straight lines in $\hh^1$, recovering the result first presented in \cite{MR2165405} and \cite{MR2481053}.

The aim of this paper is to refine the argument in \cite{GalRit15} to prove a results in the spirit of \cite{MR2583494, MR2831583} without no assumption on the existence of a sequence of smooth minimal Riemannian graphs approximating the initial sub-Riemannian minimal surface. We shall prove in Theorem~\ref{th:main}

\begin{quotation}
Let $M$ be a $3$-dimensional contact sub-Riemannian manifold, $\Om\subset M$ a domain, and $E\subset\Om$ a set of prescribed mean curvature $f\in C(\Om)$ with Euclidean Lipschitz boundary $\partial E=\Sg$. If $\Sg$ is locally a Lipschitz regular graph, then the characteristic curves are of class $C^2$.
\end{quotation}

The regularity of characteristic curves provided in Theorem~\ref{th:main} allows us to define in Section~\ref{sec:mc} a mean curvature function $H$ in the regular part of $\ptl E$, that coincides with $f$. As a consequence of the definition of the mean curvature, we shall prove in Proposition~\ref{prop:C^k} that characteristic curves are of class $C^{k+2}$ in case $f$ is of class $C^k$ when restricted to a characteristic direction. Furthermore we shall characterize the equation of characteristic curves in Proposition~\ref{prop:foliazionepergeodesiche} 

\begin{quotation}
Let $E\subset\Om$ be a set of prescribed mean curvature $f\in C(\Omega)$ with Lipschitz boundary $\Sg$ in a domain $\Om\subset M$. Then the characteristic curves in $\Sg-\Sg_0$ are $\n\text{-geodesics}$ of curvature $H$, where $H$ is the mean curvature function and $\n$ the pseudo-hermitian connection. 
\end{quotation}
We remark that $\n\text{-geodesics}$ are not in general sub-Riemannian geodesics, see Remark~\ref{rem:geodeticheaconfronto}, and examples of area-stationary surfaces foliated by $\n\text{-geodesics}$ that are not sub-Riemannian geodesics can be found in \cite{MR3259763}.  We note that Proposition~\ref{prop:foliazionepergeodesiche} provides a characterization of characteristic curves in general 3-dimensional manifolds. This characterization  is more natural than the one presented in \cite{MR2165405}, \cite{MR2481053} and \cite{gr-aim}. Indeed the property that a minimal surface is foliated by horizontal straight lines is not preserved by contact changes of coordinates. In the Heisenberg structure used in \cite{MR2583494} minimal surfaces are foliated by horizontal parabolas, that are  $\n\text{-geodesics}$ with vanishing curvature. We also stress that the structure used in \cite{MR2583494}, see also \cite{MR2831583}, is motivated by application in image reconstruction and vision problems, see for example \cite{MR2235475}, \cite{MR2366414} and \cite{MR3077550}. 
Finally it is worth of mention that our results are optimal in view of \cite[Example~2.8]{MR2333095}.

The paper is organized as follows. In Section~\ref{sec:preliminaries} we provide the necessary background on contact sub-Riemannian manifolds and sets of locally finite perimeter with prescribed mean curvature. In Section~\ref{sec:lipregsurf} we describe Euclidean Lipschitz surfaces that are intrinsic regular graphs. The main result, Theorem~\ref{th:main}, is proven in Section~\ref{sec:main}. The consequences on the mean curvature, higher regularity and a geometric characterization for characteristic curves will appear in Section~\ref{sec:mc}.


\section{Preliminaries}\label{sec:preliminaries}

\subsection{Contact sub-Riemannian manifolds}\label{sec:2.1}

Let $M$ be a $3$-dimensional smooth manifold $M$ equipped with a contact form $\omega$ and a sub-Riemannian metric $g_\mathcal{H}$ defined on its \emph{horizontal distribution} $\mathcal{H}:=\text{ker}(\omega)$.  By definition, $d\omega|_{\mathcal{H}}$ is non-degenerate. We shall refer to $(M,\omega,g_\mathcal{H})$ as a \emph{$3$-dimensional contact sub-Riemannian manifold}. It is well-known that $\omega\wedge d\omega$ is an orientation form in $M$. Since
\[
d\omega(X,Y)=X(\omega(Y))-Y(\omega(X))-\omega([X,Y]),
\]
the horizontal distribution $\mathcal{H}$ is completely non-integrable. 
\begin{definition}\label{def:horizontal} A vector field $X$ is called \emph{horizontal} if $X\in \hhh$ . A \emph{horizontal curve} is a $C^1$ curve whose tangent vector lies in the horizontal distribution.
\end{definition}
The \emph{Reeb vector field} $T$ in $M$ is the only one satisfying
\begin{equation}
\label{eq:reeb}
\omega(T)=1,\qquad \mathcal{L}_T\omega=0,
\end{equation}
where $\mathcal{L}$ is the Lie derivative in $M$.

A canonical contact structure in Euclidean $3$-space $\rr^3$ with coordinates $(x,y,t)$ is given by the contact one-form $\omega_0:=dt-ydx$. The associated contact manifold is the Heisenberg group $\hh^1$. A basis of the horizontal left invariant vector is 
\begin{equation}\label{X_0andY_0}
X_0=\frac{\partial}{\partial x} +y\frac{\partial}{\partial t}, \quad Y_0=\frac{\partial}{\partial y},
\end{equation}
while the Reeb vector field is 
\begin{equation}\label{X_0andY_0}
T_0=\frac{\partial}{\partial t}.
\end{equation}

 Darboux's Theorem \cite[Thm.~3.1]{MR1874240} (see also \cite{MR2979606}) implies that, given a point $p\in M$, there exists an open neighborhood $U$ of $p$ and a diffeomorphism $\phi_p$ from $U$ into an open set of $\rr^{3}$ satisfying $\phi_p^*\omega_0=\omega$. Such a local chart will be called a \emph{Darboux chart}. Composing the map $\phi_p$ with a contact transformation of $\hh^1$ also provides a Darboux chart. This implies we can prescribe the image of a point $p\in U$ and the image of a horizontal direction in $T_pM$.

The metric $g_\mathcal{H}$ can be uniquely extended to a Riemannian metric $g$ on $M$ by requiring $T$ to be a unit vector orthogonal to $\mathcal{H}$. The Levi-Civita connection associated to $g$ will be denoted by $D$. 


The Riemannian volume element in $(M,g)$ will be denoted by $dM$ and coincides with Popp's measure \cite[\S~10.6]{MR1867362}, \cite{MR3108867}. The volume of a set $E\subset M$ with respect to the Riemannian metric $g$ will be denoted by $|E|$. 

\subsection{Torsion and the sub-Riemannian connection}

The following is taken from \cite[\S~3.1.2]{Gaphd}. In a contact sub-Riemannian manifold, we can decompose the endomorphism $X\in TM\rightarrow D_X T$ into its antisymmetric and symmetric parts, which we will denoted by $J$ and $\tau$, respectively,
\begin{equation}\label{eq:jtau}
\begin{split}
2\escpr{J(X),Y}&=\escpr{D_X T,Y}-\escpr{D_Y T,X},\\
2\escpr{\tau(X),Y}&=\escpr{D_X T,Y}+\escpr{D_Y T,X}.
\end{split}
\end{equation}
Observe that $J(X),\tau(X)\in\hhh$ for any vector field $X$, and that $J(T)=\tau(T)=0$. Also note that
\begin{equation}\label{eq:J}
2\escpr{J(X),Y}=-\escpr{[X,Y],T}, \qquad X,Y\in\mathcal{H}.
\end{equation}
We will call $\tau$ the \emph{$($contact$)$ sub-Riemannian torsion}. We note that our $J$ differs from the one defined in \cite[(2.4)]{MR3044134} by the constant $g([X,Y],T)$, but plays the same geometric role and can be easily generalized to higher dimensions, \cite[\S~3.1.2]{Gaphd}. 

Now we define the \emph{$($contact$)$ sub-Riemannian connection} $\nabla$ as the unique metric connection, \cite[eq.~(I.5.3)]{MR2229062}, with torsion tensor $\tor(X,Y)=\nabla_XY-\nabla_YX-[X,Y]$ given by
\begin{equation}\label{def:tor}
\tor(X,Y):=\escpr{X,T}\,\tau(Y)-\escpr{Y,T}\,\tau(X)+2\escpr{J(X),Y}\,T.
\end{equation}
From \eqref{def:tor} and Koszul formula for the connection $\nabla$ it follows that $\nabla T\equiv 0$.


The standard orientation of $M$ is given by the $3$-form $\omega\wedge d\omega$. If $X_p$ is horizontal, then the basis $\{X_p,J(X_p),T_p\}$ is positively oriented, in fact $(\omega\wedge d\omega)(X,J(X),T)$ and $d\omega(X,J(X))$ have the same sign and
\[
d\omega(X,J(X))=-\omega([X,J(X)])=-g([X,J(X)],T)=g(\tor(X,J(X)),T)=2\,g(J(X),J(X))>0.
\]

\subsection{Perimeter and Lipschitz surfaces}

A set $E\subset M$ has \emph{locally finite perimeter} if, for any bounded open set $B\subset M$, we have
\[
P(E,B):=\sup\bigg\{\int_{E\,\cap\,B}\divv U\,dM: U\ \text{horizontal},\ \text{supp}(U)\subset B, ||U||_\infty\le 1\bigg\}<+\infty.
\]
The quantity $P(E,B)$ is called the \emph{relative perimeter} of $E$ in $B$.

Assuming $\ptl E$ is a Lipschitz surface and using the Divergence Theorem for Lipschitz boundaries, it is easy to show that the relative perimeter of $E$ in a bounded open set $B\subset M$ coincides with the sub-Riemannian area of $\ptl E\cap B$, given by
\begin{equation}
\label{eq:areaC1}
A(\ptl E\cap B)=\int_{\ptl E\,\cap\, B}|N_h|\,d(\ptl E).
\end{equation}
Here $N$ is the Riemannian unit normal to $\ptl E$, $N_h$ is the horizontal projection of $N$ to the horizontal distribution, and $d(\ptl E)$ is the Riemannian area measure, all computed with respect the Riemannian metric $g$, see \cite{MR2312336}. 

\subsection{Sets with prescribed mean curvature}
\label{sec:pmc}
Let $\Om$ a domain in $M$ and let $f:\Om\to\rr$ a given function. We shall say that a set of locally finite perimeter $E\subset\Om$ has \emph{prescribed mean curvature $f$ on $\Om$} if, for any bounded open set $B\subset\Om$, $E$ is a critical point of the functional
\begin{equation}
\label{eq:a-hv}
P(E,B)-\int_{E\cap B} f,
\end{equation}
where $P(E,B)$ is the relative perimeter of $E$ in $B$, and the integral on $E\cap B$ is computed with respect to the canonical Popp's measure on $M$. The admissible variations for this problem are the flows induced by vector fields with compact support in $B$.

Let $\Sg=\ptl E$ be a Lipschitz surface in $\Om$, we say that $\Sg$ has \emph{prescribed mean curvature} $f$ if it is a critical point of the functional
\begin{equation}\label{eq:prescribedfunctional}
A(\Sg\cap B) -\int_{E\cap B} f,
\end{equation}
for any bounded open set $B\subset\Om$.

These definitions are the counterpart of the ones in the Euclidean setting, \cite[(12.32) and Remark~17.11]{MR2976521}. We note that they are been introduced in the sub-Riemannian setting in \cite{GalRit15} . 

\begin{remark}If $E$ is a critical point of the relative perimeter $P(E,B)$ in any bounded open set $B\subset\Om$, then $E$ has vanishing prescribed mean curvature.
\end{remark}

\section{A characterization of Euclidean Lipschitz regular surfaces}\label{sec:lipregsurf}

The notion of \emph{ regular surface} in the sub-Riemannian geometry has been introduced by Franchi, Serapioni and Serra Cassano, \cite{MR1871966, MR1984849}.  They defined a regular surface $\Sg$ as the zero level set of a function $f$, whose horizontal gradient $\nabla^h(f)$ never vanishes. The \emph{horizontal gradient} is defined as
\[
\nabla^h(f):=X_1(f)\,X_1+X_2(f)\,X_2,
\]
where $\{X_1,X_2\}$ is an orthonormal basis of $\hhh$. We remark that the function $f$ defining a regular surface $\Sg$ can be merely $\frac{1}{2}$-Hölder continuous from the Euclidean viewpoint and there exists regular surfaces with fractional Euclidean Hausdorff measure, \cite{MR2124590}.
This notion can not be applied to Euclidean Lipschitz surfaces since $\nabla^h(f)$ is not defined at every point. On other hand it is well-known that a regular surface can be locally described as an \emph{intrinsic graph}, using an implicit function theorem, see \cite{MR1871966, MR2313532, MR2263950}.

Since we are mainly focused on local regularity property, it is not restrictive to describe an Euclidean Lipschitz surface $\Sg$ as follows. 
Given a point $p\in\Sg$, we consider a Darboux chart $(U_p,\phi_p)$ such that $\phi_p(p)=0$. The metric $g_{\mathcal{H}}$ can be described in this local chart by the matrix of smooth functions
\[
G=\begin{pmatrix}
g_{11} & g_{12}
\\
g_{21} & g_{22} 
\end{pmatrix}
=
\begin{pmatrix}
g(X,X) & g(X,Y)
\\
g(Y,X) & g(Y,Y)
\end{pmatrix}.
\]
Under the assumption that $\Sg$ is an Euclidean Lipschitz surface, we can define $p$ a \emph{regular point} in $\Sg$ if there exists (eventually after a Euclidean rotation around the $t$-axis) an open neighborhood $B\cap\Sg$ that is the graph $G_u$ of a Euclidean Lipschitz function $u:D\to\rr$ defined on a domain $D$ in the vertical plane $y=0$. A point $p\in\Sg$ is \emph{singular} when it is not regular and we denote by $\Sg_0$ the set of the singular points of $\Sg$. We remark that an analogous notion was used in 3-dimensional Lie groups in \cite{MR2831583}.

Now we introduce some geometric quantities of the graph $G_u$. We have that $G_u$ can be parameterized by the map $u:D\to\rr^3$ defined by
\[
f_u(x,t):=(x,u(x,t),t),\qquad (x,t)\in D.
\]
The tangent plane is defined in a.e. point in $G_u$ and, when it exists, it is generated by the vectors $E_1,E_2$ obtained in the following way
\begin{align*}
\tfrac{\ptl}{\ptl x}&\mapsto E_1=(1,u_x,-u-xu_x)=X+u_xY-uT,
\\
\tfrac{\ptl}{\ptl t}&\mapsto E_2=(0,u_t,1-xu_t)=u_tY+T.
\end{align*}
Then the characteristic direction is given by $Z=\widetilde{Z}/|\widetilde{Z}|$, where
\[
\widetilde{Z}=X+(u_x+uu_t)Y.
\]
We observe that the vector $\widetilde{Z}$ is well-defined and continuous in an Euclidean Lipschitz regular  surface since $G_u$ is  an intrinsic graph, see \cite{MR2223801, MR2600502}. 

If $\ga(s)=(x(s),t(s))$ is a $C^1$ curve in $D$, then the lifted curve in $G_u$ is
\[
\Ga(s)=(x(s),u(x(s),t(s)),t(s)-x(s)u(x(s),t(s)))\subset G_u.
\]
We stress that even if $\Ga$ is Lipschitz and the tangent vector
\[
\Ga'(s)=x'\,(X+u_xY-uT)+t'\,(u_tY+T)=x'X+(x'u_x+t'u_t)Y+(t'-ux')T
\]
exists for a.e. $s$, the vertical component of $\Ga'(s)$ is well defined in all point of $G_u$. 
In particular, horizontal curves in $G_u$ satisfy the ordinary differential equation $t'=ux'$. 

\begin{remark} Since $u\in Lip(D)$, we have uniqueness of characteristic curves through any given point in $G_u$.
\end{remark}

Finally, from \cite[(4.2)]{GalRit15} and the approximation technique in \cite[Theorem~1.7]{MR3168633}, we obtain
\begin{equation}
\label{eq:areagraph}
A(G_u)=\int_D \big(g_{22}(f_u)(u_x+uu_t)^2+2g_{12}(f_u)(u_x+uu_t)+g_{11}(f_u)\big)^{1/2}dxdt.
\end{equation}
We remark that the functions $g_{ij}$ are bounded on the Darboux chart $U$.

\section{Regularity of prescribed mean curvature sets with Lipschitz boundary}
\label{sec:main}

First we show a characterization on the area-stationary property of $G_u$

\begin{proposition}
\label{prop:1stvariation}
Let $G_u$ be the graph of an Euclidean Lipschitz function $u:D\to\rr$, defined on a Darboux chart by   the parametrization 
\[
f_u(x,t):=(x,u(x,t),t),\qquad (x,t)\in D.
\]
We suppose that the area of $G_u$ is defined by \eqref{eq:areagraph} and $G_u$ has prescribed mean curvature $f$. Then  
\begin{equation}\label{eq:1stvariation smooth test function}
\int_D \big(Kv+M\,(v_x+uv_t+vu_t\big)\,dxdt=0
\end{equation}
for any $v\in C^1_0(D)$, where the functions $K$, $K_1$ and $M$ are given by
\[
K=K_1-f\,\det(G),
\]
\[
K_1=\frac{\frac{\partial g_{22}}{\partial y}\cdot(u_x+uu_t)^2+2 \frac{\partial g_{12}}{\partial y}\cdot(u_x+uu_t)+\frac{\partial g_{11}}{\partial y}}{2(g_{22}\cdot(u_x+uu_t)^2+2g_{12}(u_x+uu_t)+g_{11})^{1/2}},
\]
and
\begin{equation}\label{def:M}
M=\frac{g_{22}(u_x+uu_t)+g_{12}}{(g_{22}(u_x+uu_t)^2+2g_{12}(u_x+uu_t)+g_{11})^{1/2}}.
\end{equation}
\end{proposition}

\begin{proof} First we prove a first variation formula for $A(G_u)$ for variations of $G_u$ given by $G_{u+sv}$, where  $v\in C^1_0(D)$ and $s$ is a small parameter.  Denoting $w=u+sv$, we have
\[
\frac{d}{ds}\bigg|_{s=0}A(G_{w})= \int_D \frac{d}{ds}\bigg|_{s=0} \big(g_{22}(f_w)(w_x+ww_t)^2+2g_{12}(f_w)(w_x+ww_t)+g_{11}(f_w)\big)^{1/2}dxdt.
\]
Taking into account that
\[
\frac{d}{ds} g_{ij}(f_w)=\frac{\partial}{\partial y}g_{ij} (f_w)\cdot  \frac{d}{ds}w=  \frac{\partial}{\partial y}g_{ij} (f_w)\cdot v
\]
and
\[
\frac{d}{ds} (w_x+ww_t)= v_x+uv_t+vu_t,
\]
we conclude 
\[
\frac{d}{ds}\bigg|_{s=0}A(G_{w})=\int_D \big(K_1v+M\,(v_x+uv_t+vu_t\big)\,dxdt=0.
\]
Here $K$ and $M$ are defined in the statement of the proposition. On the other hand it is easy to see that
\[
\frac{d}{ds}\bigg|_{s=0} \int_{\text{subgraph}\,G_{u+sv}} f=\int_D f\,\det(G)\,v\,dxdt,
\]
which conclude the proof.
\end{proof}

Now we are able to prove the main regularity result

\begin{theorem}
\label{th:main}
Let $M$ be a $3$-dimensional contact sub-Riemannian manifold, $\Om\subset M$ a domain, and $E\subset\Om$ a set of prescribed mean curvature $f\in C(\Om)$ with Euclidean Lipschitz boundary $\partial E=\Sg$. If $\Sg$ is locally a Lipschitz regular graph, then the characteristic curves are of class $C^2$.
\end{theorem}

\begin{proof} Given any point $p\in \Sg-\Sg_0$, it is sufficient show the result in the Darboux chart $(U,\phi)$ introduced in Section~\ref{sec:lipregsurf}. We consider
\[
Z=\frac{X+(u_x+uu_t)\,Y}{(g_{22}(u_x+uu_t)^2+2g_{12}(u_x+uu_t)+g_{11})^{1/2}},
\]
the tangent vector to an horizontal curve in $\G_u$. The function $M$ introduced in \eqref{def:M} coincides with $g(Z,Y)\circ f_u$. Then
\[
1=|Z|^2=\det(G)^{-1}\big(g_{22}\escpr{Z,X}^2-2g_{12}\escpr{Z,X}\escpr{Z,Y}+g_{11}\escpr{Z,Y}^2\big)
\]
and
\begin{equation}
\label{eq:ZX}
\escpr{Z,X}=\frac{g_{12}g(Z,Y)\pm (\det(G)(g_{22}-\escpr{Z,Y}^2))^{1/2}}{g_{22}}.
\end{equation}
By Schwarz's inequality $\escpr{Z,Y}^2\le \escpr{Y,Y}=g_{22}$. Inequality is strict since otherwise $Y$ and $Z$ would be collinear. Hence $\escpr{Z,X}$ has the same regularity as $\escpr{Z,Y}$ by \eqref{eq:ZX}.

Now we proceed as in the proof of \cite[Theorem~4.1]{GalRit15}, see also \cite[Theorem~3.5]{gr-aim}. Assume the point $p\in G_u$ corresponds to the point $(a,b)$ in the $xt$-plane. The  curve $s\mapsto (s,t(s))$ is (a reparameterization of the projection of) a characteristic curve if and only if the function $t(s)$ satisfies the ordinary differential equation $t'(s)=u(s,t(s))$. For $\eps$ small enough, we consider the solution $t_\eps$ of equation $t_\eps'(s)=u(s,t_\eps(s))$ with initial condition $t_\eps(a)=b+\eps$, and define $\ga_\eps(s):=(s,t_\eps(s))$, with $\ga=\ga_0$. We may assume that, for small enough $\eps$, the functions $t_\eps$ are defined in the interval $[a-r,a+r]$ for some $r>0$. 

We consider the parameterization
\[
F(\xi,\eps):= (\xi,t_\eps(\xi))=(s,t)
\]
near the characteristic curve through $(a,b)$. The jacobian of this parameterization is given by
\[
\det\begin{pmatrix}
1 & t'_\eps
\\
0 & \ptl t_\eps/\ptl \eps
\end{pmatrix}=\frac{\ptl t_\eps}{\ptl\eps},
\]
which is positive because of the choice of initial condition for $t_\eps$ and the fact that the curves $\ga_\eps(s)$ foliate a neighborhood of $(a,b)$. We remark that $t_\eps$ is a Lipschitz function with respect to $\eps$, \cite[Theorem~2.9]{MR2961944}, and the jacobian of the parametrization is well-defined. Using the implicit function theorem for Lipschitz functions, \cite{MR0425047, MR2206097}, any function $\varphi$ can be considered as a function of the variables $(\xi,\eps)$ by making $\tilde{\varphi}(\xi,\eps):=\varphi(\xi,t_\eps(\xi))$. Changing variables, and assuming the support of $\varphi$ is contained in a sufficiently small neighborhood of $(a,b)$, we can express the integral \eqref{eq:1stvariation smooth test function} as
\[
\int_{I}\bigg\{\int_{a-r}^{a+r}\bigg(K\tilde{\varphi}+M\,\bigg(\frac{\ptl\tilde{\varphi}}{\ptl\xi}+2\tilde{\varphi} \tilde{u}_t\bigg)\bigg)\,\frac{\ptl t_\eps}{\ptl\eps}\,d\xi\bigg\}\,d\eps,
\]
where $I$ is a small interval containing $0$. Instead of $\tilde{\varphi}$, we can consider the function $\tilde{\zeta}=\tilde{\varphi} h /( t_{\eps+h}-t_\eps)$, where $h$ is a sufficiently small real parameter. We get that 
\[
\frac{ \ptl \tilde{\zeta}}{\ptl\xi}=\frac{\ptl\tilde{\varphi}}{\ptl\xi}\cdot \frac{h}{t_{\eps+h}-t_\eps}-2\tilde{\varphi}\cdot \frac{\tilde{u}(\xi, \eps+h)-\tilde{u}(\xi,\eps)}{t_{\eps+h}-t_\eps}\cdot \frac{h}{t_{\eps+h}-t_\eps}
\]
tends a.e. to 
\[
\frac{\ptl\tilde{\varphi}/\ptl\xi}{\ptl t_\eps/\ptl\eps}-\frac{2\tilde{\varphi}\tilde{u}_t}{\ptl t_\eps/\ptl\eps},
\]
when $h\rightarrow 0$. So using that $G_u$ is area-stationary we have that
\[
\int_{I}\bigg\{\int_{a-r}^{a+r} \frac{h}{t_{\eps+h}-t_\eps}\bigg(K\,  \tilde{\varphi}+M\,\bigg(\frac{\ptl\tilde{\varphi}}{\ptl\xi}\cdot +2  \tilde{\varphi}\cdot \bigg(\tilde{u}_t-\frac{\tilde{u}(\xi, \eps+h)-\tilde{u}(\xi,\eps)}{t_{\eps+h}-t_\eps}\bigg)\bigg)\bigg)\,\frac{\ptl t_\eps}{\ptl\eps}\,d\xi\bigg\}\,d\eps
\]
vanishes. Furthermore, letting $h\rightarrow 0$ we conclude
\[
\int_{I}\bigg\{\int_{a-r}^{a+r}\bigg(K\tilde{\varphi}+M\,\frac{\ptl\tilde{\varphi}}{\ptl\xi}\bigg)\,d\xi\bigg\}\,d\eps=0.
\]


Let now $\eta:\rr\to\rr$ be a positive function with compact support in the interval $I$ and consider the family $\eta_\rho(x):=\rho^{-1}\eta(x/\rho)$. Inserting a test function of the form $\eta_\rho(\eps)\psi(\xi)$, where $\psi$ is a $C^\infty$ function with compact support in $(a-r,a+r)$, making $\rho\to 0$, and using that $G_u$ is area-stationary we obtain
\[
\int_{a-r}^{a+r} \big(K(0,\xi)\,\psi(\xi)+M(0,\xi)\,\psi'(\xi)\big)\,d\xi=0
\]
for any $\psi\in C^\infty_0((a-r,a+r))$. By \cite[Lemma~4.2]{GalRit15}, the function $M(0,\xi)$, which is the restriction of $\escpr{Z,Y}$ to the characteristic curve, is a $C^1$ function on the curve. By equation \eqref{eq:ZX}, the restriction of $\escpr{Z,X}$ to the characteristic curve is also $C^1$. This proves that horizontal curves are of class $C^2$.
\end{proof}

\section{The mean curvature and a characterization of characteristic curves}
\label{sec:mc}

In this section we introduce the definition of the mean curvature of a  regular Lipschitz boundary and present some consequences. This results partially follows \cite[\S~4]{GalRit15} and we will omit some proofs.

Given a Lipschitz surface $\Sg\subset M$ such that the  characteristic curves are $C^2$-smooth in $\Sg\setminus\Sg_0$, we define the \emph{mean curvature} of $\Sg$ at $p\in \Sg\setminus\Sg_0$ by
\begin{equation}
\label{eq:Hstrongsense}
H(p):= g(\n_Z Z,\nu_h)(p).
\end{equation}
This is the standard definition of mean curvature for Euclidean $C^2$ surfaces, see e.g. \cite[(2.1)]{MR2165405}, \cite[(4.8)]{MR2435652} and \cite[(1.4)]{MR2354992}. Using the regularity  Theorem~\ref{th:main} we get

\begin{proposition}
\label{prop:1stareaStrong}
Let $\Omega\subset M$ be a domain and $E\subset \Omega$ a set of prescribed mean curvature $f\in C(\Omega)$ with Lipschitz boundary $\Sg$ with $H\in L^1_{loc}(\Sg)$. Then the first variation of the functional \eqref{eq:prescribedfunctional} induced by an horizontal vector field $U\in C^1_0(\Omega)$ is given by 
\begin{equation}
\label{eq:1stvariationprescribedfunctional}
\begin{split}
\int_{\Sg}H\, \mnh\,g(U,\nu_h) \,d\Sg-\int_\Sg f \,\mnh\,g(U,\nu_h)\,d\Sg.
\end{split}
\end{equation}
\end{proposition}

\begin{proof}
We first observe that formula
\[
\frac{d}{ds}\bigg|_{s=0} A(\Sg_s)=-\int_{\Sg}H\, \mnh\,g(U,\nu_h) \,d\Sg
\] 
can be proved as in \cite[\S~6]{MR3044134} with a standard approximation argument. Note that even $N$ is defined a.e. on $\Sg$, $\mnh$ and $\nu_h$ are well defined on $\Sg-\Sg_0$. On the other hand, it is well-known that  
\[
\frac{d}{ds}\bigg|_{s=0} \bigg(\int_{\phi_s(\Omega)} f \bigg)= -\int_\Sg f \,\mnh\,g(U,\nu_h)\, d\Sg,
\]
see e.g. \cite[17.8]{MR2976521}.
\end{proof}

The following corollary is an immediate consequence of \eqref{eq:1stvariationprescribedfunctional}

\begin{corollary} Let $E\subset\Om$ be a set of prescribed mean curvature $f\in C(\Omega)$ with Lipschitz boundary $\Sg$ in a domain $\Om\subset M$. Assume $H\in L^1_{loc}(\Sg)$. Then $H(p)=f(p)$ for any $p\in\Sg\setminus\Sg_0$. 
\end{corollary}

Finally we can improve the regularity of Theorem~\ref{th:main} assuming an higher regularity for $f$. 

\begin{proposition}
\label{prop:C^k}
Let $E\subset\Om$ be a set of prescribed mean curvature $f\in C(\Omega)$ with $C^1$~boundary $\Sg$ in a domain $\Om\subset M$. Assume that $f$ is also $C^k$ in the $Z$-direction, $k\ge 1$. Then the characteristic curves of $\Sg$ are of class $C^{k+2}$ in $\Sg\setminus\Sg_0$. 
\end{proposition}

The proof of last proposition is the analogous of \cite[Proposition~5.3]{GalRit15}. A remarkable consequence of Proposition~\ref{prop:C^k} is that critical points of perimeter (eventually under a volume constraint) with Lipschitz boundaries are foliated by smooth horizontal curves on its regular set.

To provide a geometric characterization of characteristic curves we need the following definition

\begin{definition} A curve $\ga$ is a $\n$-geodesic of curvature $h:I\rightarrow \rr$ if 
\[
\n_{\dot{\ga}}\dot{\ga}+h J(\dot{\ga})=0.
\]
\end{definition} 

\begin{remark}\label{rem:geodeticheaconfronto} We stress that a $\n$-geodesic of curvature $h$ is a sub-Riemannian geodesic if and only if it satisfies 
\[
\dot{\ga}(h)=g(\tau(\dot{\ga}),\dot{\ga}),
\]
see \cite[Proposition~15]{Ru}. 
\end{remark}

\begin{proposition}\label{prop:foliazionepergeodesiche} Let $E\subset\Om$ be a set of prescribed mean curvature $f\in C(\Omega)$ with Lipschitz boundary $\Sg$ in a domain $\Om\subset M$. Then the characteristic curves in $\Sg-\Sg_0$ are $\n$-geodesic of curvature $H$, where $H$ is defined in \eqref{eq:Hstrongsense}. 

\end{proposition}

\begin{proof} Let $Z$ the tangent unit vector to a characteristic curve $\ga:I\rightarrow \Sg-\Sg_0$. Then we have
\[
\begin{split}
\n_Z Z&=g(\n_Z Z,Z)+g(\n_Z Z,T)+g(\n_Z Z,\nu_h)\\
&=g(\n_Z Z,\nu_h)\\
&=-HJ(Z).
\end{split}
\]

\end{proof}

\begin{remark}The result in Proposition~\ref{prop:foliazionepergeodesiche} was shown for $C^2$ CMC surfaces in $\hh^1$, \cite{MR2648078, MR2354992, MR2435652}, in Sasakian space forms, \cite{MR2875642}, and in general 3-dimensional contact manifolds, \cite{MR3044134, MR2401420, MR2579306, MR3259763}.

\end{remark}

\bibliography{lipschitz}

\end{document}